\documentclass[12pt,reqno]{amsart}
\usepackage{amsmath,amssymb,amsthm,mathtools}
\usepackage{enumerate}
\usepackage{geometry}
\usepackage[colorlinks=true,linkcolor=blue,citecolor=blue,urlcolor=blue]{hyperref}
\usepackage{enumitem}
\usepackage{cleveref}
\usepackage{booktabs}
\usepackage{array}
\newtheorem{theorem}{Theorem}[section]
\newtheorem{lemma}[theorem]{Lemma}
\newtheorem{proposition}[theorem]{Proposition}

\theoremstyle{definition}
\newtheorem{definition}[theorem]{Definition}
\newtheorem{remark}[theorem]{Remark}
\newtheorem{assumption}{Assumption}
\newtheorem{conjecture}[theorem]{Conjecture}
\newtheorem{openproblem}{Open problem}

\newcommand{\R}{\mathbb{R}}

\newcommand{\C}{\mathbb{C}}
\newcommand{\Lnorm}[1]{\|{#1}\|_{L^{2}(\Omega)}}

\newcommand{\dpt}[1]{\partial_{t}^{#1}}

\begin{document}

\title[Large-Time Upper Decay Estimates for Weakly Coupled Systems]{Large-Time Upper Decay Estimates for Weakly Coupled Time-Fractional Diffusion Systems}

\author[H.~Nguyen, T.~Nguyen, and I.~Witt]{Huong Nguyen, Thang Nguyen, and Ingo Witt}
\address{Faculty of Mathematics and Informatics, Hanoi University of Science and Technology}
\email{huong.nguyenthithu3@hust.edu.vn}

\address{Faculty of Mathematics and Informatics, Hanoi National University of Education}
\email{thangnn@hnue.edu.vn}

\address{Mathematical Institute
University of G\"ottingen, Germany}
\email{iwitt@mathematik.uni-goettingen.de}

\begin{abstract}
We prove conjectured-optimal large-time upper decay estimates for
solutions of weakly coupled systems of $m\ge2$ time-fractional
diffusion equations. Our result establishes the upper-bound assertion
in the conjecture of Li, Liu, and Wada~\cite{LLW26}. 
The sharpness of these upper exponents for $m\ge3$ remains open in
general; see Section~\ref{sec:discussion}.

Assume that the diffusion matrices are uniformly elliptic, the coupling
coefficients are bounded and may depend on time, and the matrix of $L^\infty$-norms of the coupling coefficients
satisfies a spectral-radius stability condition. Let $\alpha^*$ denote
the smallest fractional order among the components with nonzero initial
data. Then the $L^2(\Omega)^m$-norm of the solution is
$O(t^{-\alpha^*})$ as $t\to\infty$ if $\alpha^*<1$, whereas it is
$O(t^{-(1+\alpha_m)})$ if $\alpha^*=1$, where $\alpha_m$ is the
smallest fractional order occurring in the system. For $m=2$, 
the result yields the same upper decay exponents as
\cite[Theorem~1.1]{LLW26}; in the absence of self-coupling, the
matrix stability condition is weaker than the corresponding
small-coupling condition in that work. The case $m\ge3$ is new.

Our proof uses a large-time weighted Volterra framework. If
$\alpha^*<1$, the weight is a regularly varying power weight; if
$\alpha^*=1$, it is constructed by convolution with the resolvent
kernel associated with the smallest fractional order. These weights
yield the required large-time bounds for all resolvent kernels in the
system. In particular, the argument avoids the Laplace-transform and
residue-calculus methods used in \cite{LLW26}, which become difficult
to implement for systems with three or more components.
\end{abstract}
\date{September 2026}
\maketitle

\section{Introduction}
\label{sec:intro}

\subsection{Background and motivation}

The study of anomalous diffusion---non-Gaussian \linebreak transport in complex
media such as porous materials and biological tissues---has stimulated
extensive research on time-fractional partial differential equations.
For a scalar diffusion equation of fractional order $\alpha\in(0,1)$,
the solution typically decays polynomially at the rate $t^{-\alpha}$,
in sharp contrast with the exponential decay associated with classical
parabolic equations; see, for example, \cite{SY11,VZ15}. Coupled systems
exhibit richer dynamics: interactions among components with distinct
fractional orders can produce qualitative large-time behaviors that do
not arise in scalar problems.

Recently, Li, Liu, and Wada \cite{LLW26} investigated the two-component
system
\begin{equation}\label{eq:LLW-system}
  \begin{cases}
    \dpt{\alpha}(u-u_0)-\mathrm{div}(\mathbf{A}\nabla u)
      +c_{11}u+c_{12}v=0, \\
    \dpt{\beta}(v-v_0)-\mathrm{div}(\mathbf{B}\nabla v)
      +c_{21}u+c_{22}v=0,
  \end{cases}
  \quad\text{in }\Omega\times\R_+,
\end{equation}
with $1\!\ge\!\alpha\!>\!\beta\!>\!0$ and $v_0\!\equiv\!0$. They proved that the
solution decays at the rate $t^{-\alpha}$ when $\alpha<1$, whereas it
decays at the faster rate $t^{-(1+\beta)}$ when $\alpha=1$. The latter
case is particularly striking: coupling to a classical diffusion
component causes the fractional component to decay faster than any
scalar time-fractional diffusion equation. This superlinear decay
phenomenon has no analogue in scalar models.

Motivated by extensive numerical experiments for systems with $m$
components, Li, Liu, and Wada formulated the following conjecture.

\begin{conjecture}[Li--Liu--Wada \cite{LLW26}]\label{conj:LLW}
Consider system \eqref{eq:general-system} below with $m\ge2$
components and orders $1\ge\alpha_1\ge\dotsc\ge\alpha_m>0$, $\alpha_m<1$.
Let $\alpha^{*}\coloneq\min\{\alpha_i\colon u_i^0\not\equiv 0\}$.
Under suitable stability assumptions, there exists a $C>0$ such that
\[
  \sum_{i=1}^m\Lnorm{u_i(t)}\le
  \begin{cases}
    C\,t^{-\alpha^{*}}, & \alpha^{*}<1,\\
    C\,t^{-(1+\alpha_m)}, & \alpha^{*}=1,
  \end{cases}
  \qquad t\geq 1.
\]
\end{conjecture}

\subsection{Main difficulties and our approach}

The proof of \cite[Theorem~1.1]{LLW26} in the two-component case relies
on the Laplace transform and a detailed residue analysis on the complex
plane cut along the negative real axis. For $m\geq3$, the corresponding
characteristic relation takes the schematic form
\[
  (s^{\alpha_1}+\lambda_1)\cdots(s^{\alpha_m}+\lambda_m)
  = \text{coupling terms}.
\]
Because this relation involves several distinct fractional powers of
$s$, its roots cannot in general be analyzed explicitly. The associated
multivalued complex structure makes a residue-calculus argument
substantially more difficult, if not impractical.

We avoid this difficulty by working entirely in the \emph{time domain}.
Our argument consists of the following steps:
\begin{enumerate}[label=\textup{(\roman*)}]
  \item \textbf{Mild-solution representation.}
  Each component $u_i$ satisfies a Duhamel formula involving the
  resolvent operators $S_i(t,\mathcal{A}_i)$ and
  $R_i(t,\mathcal{A}_i)$. Their operator norms are bounded by scalar
  Mittag--Leffler kernels
  $s_{\alpha_i}(t,\lambda_i)$ and
  $r_{\alpha_i}(t,\lambda_i)$, respectively.

   \item \textbf{Large-time admissible weights.}
  We construct a positive weight $w\in\mathcal W(\mu)$, where
  \[
    \mu=
    \begin{cases}
      \alpha^*, & \alpha^*<1,\\
      1+\alpha_m, & \alpha^*=1.
    \end{cases}
  \]
  The weight controls the homogeneous terms for sufficiently large
  times and satisfies
  \[
    \int_{T_0}^t
    r_{\alpha_i}(t-s,\lambda_i)w(s)\,ds
    \leq K_iw(t),
    \qquad t\geq T_0.
  \]
  If $\alpha^*<1$, we take
  \[
    w(t)=(1+t)^{-\alpha^*}.
  \]
  If $\alpha^*=1$, we construct $w$ from the resolvent kernel
  corresponding to the smallest fractional order.

  \item \textbf{Closure under condition~\textup{(S)}.}
  The spectral-radius stability condition for the matrix of coupling
norms ensures that the matching constants can be chosen so that
  \[
     \rho(\mathsf K\mathsf C)<1,
  \] 
  where $\mathsf C
    \coloneq
    \left(
      \|c_{ij}\|_{L^\infty(\Omega\times\R_+)}
    \right)_{1\le i,j\le m}$,
    $\mathsf K
    \coloneq
    \operatorname{diag}(K_1,\ldots,K_m)$, and $\rho$ denotes the spectral radius.
  Splitting the Duhamel integral into an early-time part and a
  large-time part, a weighted Volterra argument then yields
  \[
    \|u_i(t)\|_{L^2(\Omega)}
    \leq Cw(t)
  \]
  for all sufficiently large $t$.
\end{enumerate}

\subsection{Statement of the main result}

Let $m\ge2$, let $\Omega\subset\R^d$ be a bounded domain with smooth
boundary and consider
\begin{equation}\label{eq:general-system}
  \begin{cases}
    \dpt{\alpha_i}(u_i-u_i^0)-\mathrm{div}(\mathbf{A}_i(x)\nabla u_i)
      +\displaystyle\sum_{j=1}^m c_{ij}(x,t)\,u_j=0
        & \text{in }\Omega\times\R_+,\\
    u_i=0 & \text{on }\partial\Omega\times\R_+,
  \end{cases}
  \quad i=1,\dotsc,m,
\end{equation}
where $1\ge\alpha_1\ge\cdots\ge\alpha_m>0$.
Here $\dpt{\alpha_i}$ denotes the Caputo fractional derivative of order
$\alpha_i$ when $\alpha_i<1$, and the usual time derivative when
$\alpha_i=1$.

\renewcommand{\theassumption}{(\Alph{assumption})} 
\begin{assumption}[Uniform ellipticity]\label{ass:elliptic}
  For each $i$, $\mathbf{A}_i\in C^1(\overline\Omega;\R^{d\times d}_{\mathrm{sym}})$
  and there exists a $\kappa_i>0$ such that
  $\mathbf{A}_i(x)\xi\cdot\xi\ge\kappa_i|\xi|^2$ for all $x\in\overline\Omega$,
  $\xi\in\R^d$.
\end{assumption}

\setcounter{assumption}{2}
\begin{assumption}[Bounded coupling]\label{ass:coupling}
  $c_{ij}\in L^\infty(\Omega\times\R_+)$ and $c_{ii}(x,t)\ge0$ for
  all $i$.
\end{assumption}

\begin{remark}[On the role of $c_{ii}\ge0$]\label{rem:cii}
  The condition $c_{ii}\geq0$ reflects the reaction--diffusion
  structure of \eqref{eq:general-system}: the self-interaction of each
  component is nonnegative and hence dissipative. It is the direct
  analogue of the assumptions $c_{11}\geq0$ and $c_{22}\geq0$ in
  \cite[Assumption~(C)]{LLW26}.
  This sign condition is not essential to the proof of
  Theorem~\ref{thm:main}. Indeed, the a priori estimate
  \eqref{eq:scalar-ineq} follows from the triangle inequality and the
  operator bound \eqref{eq:op-bounds}, which give
  \[
    \|R_i(t-s,\mathcal{A}_i)(-c_{ij}(s))u_j(s)\|_{L^2}
    \leq
    r_{\alpha_i}(t-s,\lambda_i)
    \|c_{ij}(\cdot,s)\|_{L^\infty}
    \|u_j(s)\|_{L^2},
  \]
  independently of the sign of $c_{ij}$, including the case $i=j$.
  Consequently, Theorem~\ref{thm:main} remains valid if the condition
  $c_{ii}\geq0$ is removed, provided that
  $\|c_{ii}\|_{L^\infty(\Omega\times\R_+)}$ is included in the coupling
  matrix $\mathsf C$ in assumption~\textup{(S)}. We impose
  $c_{ii}\geq0$ only to preserve the reaction-diffusion interpretation
  of the model.
\end{remark}

\setcounter{assumption}{18}
\begin{assumption}[Matrix stability]\label{ass:stability}
  Let $C_\Omega$ denote the Poincar\'e constant of $\Omega$, and set
  \[
    \lambda_i\coloneq\frac{\kappa_i}{C_\Omega^2},
    \qquad
    \gamma_{ij}
    \coloneq
    \|c_{ij}\|_{L^\infty(\Omega\times\R_+)},
    \qquad
    \mu_i\coloneq\sum_{j=1}^m\gamma_{ij}.
  \]
  Define
  \[
    \mathsf C\coloneq(\gamma_{ij})_{1\leq i,j\leq m},
    \qquad
    \mathsf\Lambda
    \coloneq
    \operatorname{diag}(\lambda_1,\ldots,\lambda_m).
  \]
  We assume that
  \begin{equation}\label{eq:stability}
    \rho(\mathsf\Lambda^{-1}\mathsf C)<1,
  \end{equation}
  where $\rho$ denotes the spectral radius.
\end{assumption}

\begin{remark}\label{rem_S}
  The matrix stability condition \eqref{eq:stability} is implied by
  the row-sum condition
  \[
    \lambda_i>
    \sum_{j=1}^m
    \|c_{ij}\|_{L^\infty(\Omega\times\R_+)},
    \qquad i=1,\ldots,m.
  \]
  Indeed,
  \[
    \rho(\mathsf\Lambda^{-1}\mathsf C)
    \leq
    \|\mathsf\Lambda^{-1}\mathsf C\|_{\ell^\infty\to\ell^\infty}
    =
    \max_{1\leq i\leq m}
    \frac{1}{\lambda_i}
    \sum_{j=1}^m
    \|c_{ij}\|_{L^\infty(\Omega\times\R_+)}
    <1.
  \]

  In the special two-component case
  \[
    c_{11}\equiv c_{22}\equiv0,
  \]
  let
  \[
    a\coloneq
    \|c_{12}\|_{L^\infty(\Omega\times\R_+)},
    \qquad
    b\coloneq
    \|c_{21}\|_{L^\infty(\Omega\times\R_+)}.
  \]
  Then
  \[
    \mathsf\Lambda^{-1}\mathsf C
    =
    \begin{pmatrix}
      0 & a/\lambda_1\\
      b/\lambda_2 & 0
    \end{pmatrix},
  \]
  so that \eqref{eq:stability} is equivalent to
  \[
    \rho(\mathsf\Lambda^{-1}\mathsf C)
    =
    \sqrt{\frac{ab}{\lambda_1\lambda_2}}
    <1,
  \]
  or equivalently,
  \[
    ab<\lambda_1\lambda_2.
  \]
  The condition imposed in \cite[assumption~(1.7)]{LLW26},
  \[
    \lambda_i>\max\{a,b\},
    \qquad i=1,2,
  \]
  implies this inequality.
\end{remark}

\begin{remark}
  The diagonal coefficients enter the matrix stability condition above
  because the present proof estimates the Duhamel term by the triangle
  inequality. If the coefficients $c_{ii}$ are time independent and
  nonnegative, they may instead be incorporated into the elliptic
  operators $\mathcal A_i$, so that only off-diagonal coefficients
  enter the stability matrix. Extending this refinement to
  time-dependent nonnegative diagonal coefficients would require
  resolvent estimates for nonautonomous fractional diffusion equations.
\end{remark}

\begin{theorem}[Main result]\label{thm:main}
    Suppose that assumptions \textup{(A)}, \textup{(C)}, and
  \textup{(S)} hold, and that
  \[
    1\geq\alpha_1\geq\cdots\geq\alpha_m>0,
    \qquad
    \alpha_m<1.
  \]
  Let $u_i^0\in L^2(\Omega)$ for $i=1,\dotsc,m$. If
  \[
    u_i^0\equiv0,
    \qquad i=1,\ldots,m,
  \]
  then the unique solution is identically zero. Suppose otherwise, and
  define
  \begin{equation}\label{eq:alpha-star}
    \alpha^*
    \coloneq
    \min\bigl\{\alpha_i: u_i^0\not\equiv0\bigr\}.
  \end{equation}
  Then there exists a constant $C>0$ such that the solution
  $(u_1,\dotsc,u_m)$ of \eqref{eq:general-system} satisfies
  \begin{equation}\label{eq:main-decay}
    \sum_{i=1}^m \Lnorm{u_i(t)}
    \leq
    \begin{cases}
      C\,t^{-\alpha^*}, & \alpha^*<1,\\
      C\,t^{-(1+\alpha_m)}, & \alpha^*=1,
    \end{cases}
    \qquad t\geq1.
  \end{equation}
\end{theorem}

The exponents in Theorem~\ref{thm:main} coincide with the rates
conjectured in \cite{LLW26}. In the two-component setting,
\cite{LLW26} shows through explicit examples that these exponents
cannot in general be improved. For $m\ge3$, the corresponding general
sharpness problem remains open.

\begin{remark}[Eventual matrix stability]
  The global matrix stability condition \eqref{eq:stability} may be
  weakened to an eventual condition. For $T\ge0$, define
  \[
    \mathsf C(T)
    \coloneq
    \left(
      \operatorname*{ess\,sup}_{t\ge T}
      \|c_{ij}(\cdot,t)\|_{L^\infty(\Omega)}
    \right)_{1\le i,j\le m}.
  \]
  Suppose that
  \[
    \rho\bigl(\mathsf\Lambda^{-1}\mathsf C(T_*)\bigr)<1
  \]
  for some $T_*>0$. Then the conclusion of
  Theorem~\ref{thm:main} remains valid.

  Indeed, choose the large-time threshold $T_0$ so that
  $T_0\ge T_*$. In the estimate of the Duhamel term over
  $[T_0,t]$, replace $\gamma_{ij}$ by
  \[
    \gamma_{ij}(T_*)
    \coloneq
    \operatorname*{ess\,sup}_{s\ge T_*}
    \|c_{ij}(\cdot,s)\|_{L^\infty(\Omega)}.
  \]
  The contribution from $[0,T_0]$ is absorbed into the early-time term,
  exactly as in the proof of Proposition~\ref{prop:large-time-closure}.
  The matrix closure argument then applies with $\mathsf C(T_*)$ in
  place of $\mathsf C$.  

  In particular, the eventual stability condition holds whenever
  \[
    \|c_{ij}(\cdot,t)\|_{L^\infty(\Omega)}\to0
    \qquad\text{as }t\to\infty
  \]
  for all $i,j$.
\end{remark}

\subsection{Organization of the paper}
Section~\ref{sec:prelim} collects properties of Mittag--Leffler
functions, resolvent kernels, and their convolutions.
Section~\ref{sec:wellposed} establishes well-posedness and derives the
mild-solution representation.
Section~\ref{sec:framework} introduces the large-time admissible-weight
framework and proves the weighted closure principle.
Section~\ref{sec:weights} constructs the optimal weights in the two
cases of Theorem~\ref{thm:main}.
Section~\ref{sec:proof} proves Theorem~\ref{thm:main}.
Finally, Section~\ref{sec:discussion} discusses sharpness and several
open problems.

\section{Preliminaries}
\label{sec:prelim}

\subsection{Mittag--Leffler functions}

Recall the two-parameter Mittag--Leffler function
\[
  E_{\alpha,\beta}(z)
  \coloneq
  \sum_{k=0}^{\infty}\frac{z^k}{\Gamma(\alpha k+\beta)},
  \qquad z\in\C,\quad \alpha>0,\quad \beta\in\R.
\]
We write $E_\alpha\coloneq E_{\alpha,1}$. For
$\alpha\in(0,1)$, $\lambda>0$, and $t>0$, define
\begin{align}
  s_\alpha(t,\lambda)
  &\coloneq E_\alpha(-\lambda t^\alpha),
  \label{eq:sker}\\
  r_\alpha(t,\lambda)
  &\coloneq t^{\alpha-1}E_{\alpha,\alpha}(-\lambda t^\alpha).
  \label{eq:rker}
\end{align}
For $\alpha=1$, we set
\[
  s_1(t,\lambda)=r_1(t,\lambda)\coloneq e^{-\lambda t}.
\]

\begin{lemma}[Properties of Mittag--Leffler functions]
\label{lem:ML}
Let $\alpha\in(0,1)$ and $\lambda>0$. Then the following properties
hold.
\begin{enumerate}[label=\textup{(\roman*)}]
  \item For every $t>0$,
  \[
    0<E_\alpha(-\lambda t^\alpha)\leq1
  \]
  and
  \[
    0<r_\alpha(t,\lambda)
    \leq\frac{t^{\alpha-1}}{\Gamma(\alpha)}.
  \]

  \item As $t\to+\infty$,
  \begin{align}
    s_\alpha(t,\lambda)
    &\sim
    \frac{1}{\lambda\Gamma(1-\alpha)}\,t^{-\alpha},
    \label{eq:sasy}\\
    r_\alpha(t,\lambda)
    &\sim
    \frac{\alpha}{\lambda^2\Gamma(1-\alpha)}\,t^{-\alpha-1}.
    \label{eq:rasy}
  \end{align}

  \item There exists a constant $C=C(\alpha,\lambda)>0$ such that
  \[
    r_\alpha(t,\lambda)
    \leq C t^{-\alpha-1},
    \qquad t\geq1.
  \]

  \item There exists a constant $C_D=C_D(\alpha,\lambda)>0$ such that
  \[
    r_\alpha(t,\lambda)
    \leq C_D r_\alpha(2t,\lambda),
    \qquad t\geq1.
  \]
\end{enumerate}
\end{lemma}

\begin{proof}
The assertions are standard; see, for example,
\cite[Chapter~3]{Pod99} and
\cite[Theorems~3.4.1 and~3.4.2]{GKMR14}. The asymptotic formula for
$r_\alpha$ follows from the large-argument expansion of
$E_{\alpha,\alpha}$:
\[
  t^{\alpha-1}E_{\alpha,\alpha}(-\lambda t^\alpha)
  \sim
  -\frac{1}{\lambda^2\Gamma(-\alpha)}\,t^{-\alpha-1}
  =
  \frac{\alpha}{\lambda^2\Gamma(1-\alpha)}\,t^{-\alpha-1}. \qedhere
\]
\end{proof}

\begin{lemma}[Laplace transform]\label{lem:LT}
For $\alpha\in(0,1]$, $\lambda>0$, and $s>0$,
\[
  \mathcal{L}[r_\alpha(\cdot,\lambda)](s)
  =
  \frac{1}{s^\alpha+\lambda}.
\]
In particular,
$ \displaystyle \int_0^\infty r_\alpha(t,\lambda)\,dt
  =  \frac1\lambda$.
\end{lemma}

\subsection{Resolvent identity}

The following elementary identity is central to our argument.

\begin{lemma}[Resolvent identity]\label{lem:resolvent-id}
Let $\alpha\in(0,1]$ and $\lambda>\eta>0$. Then, for every $t>0$,
\begin{equation}\label{eq:resolvent-id}
  \int_0^t r_\alpha(t-s,\lambda)r_\alpha(s,\eta)\,ds
  =
  \frac{r_\alpha(t,\eta)-r_\alpha(t,\lambda)}{\lambda-\eta}
  \leq
  \frac{r_\alpha(t,\eta)}{\lambda-\eta}.
\end{equation}
\end{lemma}

\begin{proof}
By Lemma~\ref{lem:LT},
\[
  \mathcal{L}[r_\alpha(\cdot,\lambda)
  *r_\alpha(\cdot,\eta)](s)
  =
  \frac{1}{(s^\alpha+\lambda)(s^\alpha+\eta)}.
\]
Since
\[
  \frac{1}{(s^\alpha+\lambda)(s^\alpha+\eta)}
  =
  \frac{1}{\lambda-\eta}
  \left(
    \frac{1}{s^\alpha+\eta}
    -
    \frac{1}{s^\alpha+\lambda}
  \right),
\]
the identity follows by Laplace inversion. The inequality follows from
the positivity of $r_\alpha(t,\lambda)$.
\end{proof}

\begin{lemma}[Matching]\label{lem:well-interacted}
Let $\alpha\in(0,1]$, $\lambda>\eta>0$, and let
$h\in L^1_{\mathrm{loc}}(\R_+)$ satisfy $h\geq0$. Define
\[
  w_\eta(t)
  \coloneq
  \int_0^t r_\alpha(t-\tau,\eta)h(\tau)\,d\tau.
\]
Then, for every $t>0$,
\begin{equation}\label{eq:well-interacted}
  \int_0^t r_\alpha(t-s,\lambda)w_\eta(s)\,ds
  \leq
  \frac{w_\eta(t)}{\lambda-\eta}.
\end{equation}
\end{lemma}

\begin{proof}
Since all integrands are nonnegative, Tonelli's theorem and
Lemma~\ref{lem:resolvent-id} yield
\begin{align*}
  \int_0^t r_\alpha(t-s,\lambda)w_\eta(s)\,ds
  &=
  \int_0^t h(\tau)
  \int_\tau^t
  r_\alpha(t-s,\lambda)r_\alpha(s-\tau,\eta)
  \,ds\,d\tau\\
  &=
  \frac{1}{\lambda-\eta}
  \int_0^t h(\tau)
  \bigl[
    r_\alpha(t-\tau,\eta)
    -r_\alpha(t-\tau,\lambda)
  \bigr]\,d\tau\\
  &\leq
  \frac{1}{\lambda-\eta}
  \int_0^t h(\tau)r_\alpha(t-\tau,\eta)\,d\tau\\
  &=
  \frac{w_\eta(t)}{\lambda-\eta}. \qedhere
\end{align*}
\end{proof}

\begin{remark}
Lemma~\ref{lem:well-interacted} shows that weights of the form
\[
  w_\eta=r_\alpha(\cdot,\eta)*h
\]
are stable under convolution with $r_\alpha(\cdot,\lambda)$ whenever
$\lambda>\eta$. More precisely,
\[
  r_\alpha(\cdot,\lambda)*w_\eta
  \leq
  \frac{1}{\lambda-\eta}w_\eta.
\]
This is the basic algebraic mechanism behind the construction of
admissible weights below.
\end{remark}

\section{Well-Posedness and mild solution representation}
\label{sec:wellposed}

\subsection{Elliptic operators and resolvent families}

For each $i=1,\ldots,m$, let $\mathcal A_i$ denote the Dirichlet
realization in $L^2(\Omega)$ of the elliptic operator
\[
  \mathcal A_i u
  \coloneq
  -\operatorname{div}(\mathbf A_i\nabla u).
\]
Thus,
\[
  D(\mathcal A_i)=H^2(\Omega)\cap H^1_0(\Omega),
\]
and, under Assumption~\textup{(A)}, $\mathcal A_i$ is a positive
self-adjoint operator on $L^2(\Omega)$. By the ellipticity assumption
and the Poincar\'e inequality,
\[
  (\mathcal A_i u,u)_{L^2}
  \geq
  \lambda_i\|u\|_{L^2}^2,
  \qquad
  \lambda_i=\frac{\kappa_i}{C_\Omega^2}.
\]
Consequently, $-\mathcal A_i$ generates an analytic contraction
semigroup $\{e^{-t\mathcal A_i}\}_{t\geq0}$ satisfying
\[
  \|e^{-t\mathcal A_i}\|_{\mathcal B(L^2)}
  \leq e^{-\lambda_i t},
  \qquad t\geq0.
\]

For $\alpha_i\in(0,1]$, define
\begin{equation}\label{eq:op-resolvent}
  S_i(t,\mathcal A_i)
  \coloneq
  E_{\alpha_i}(-t^{\alpha_i}\mathcal A_i),
  \qquad
  R_i(t,\mathcal A_i)
  \coloneq
  t^{\alpha_i-1}
  E_{\alpha_i,\alpha_i}(-t^{\alpha_i}\mathcal A_i),
  \qquad t>0.
\end{equation}
When $\alpha_i=1$, these operators reduce to the classical heat
semigroup:
\[
  S_i(t,\mathcal A_i)
  =
  R_i(t,\mathcal A_i)
  =
  e^{-t\mathcal A_i}.
\]

By the spectral theorem, or equivalently by the subordination
principle, the operators in \eqref{eq:op-resolvent} are bounded on
$L^2(\Omega)$ and satisfy
\begin{equation}\label{eq:op-bounds}
  \|S_i(t,\mathcal A_i)\|_{\mathcal B(L^2)}
  \leq
  s_{\alpha_i}(t,\lambda_i),
  \qquad
  \|R_i(t,\mathcal A_i)\|_{\mathcal B(L^2)}
  \leq
  r_{\alpha_i}(t,\lambda_i),
  \qquad t>0.
\end{equation}

\subsection{Mild solutions and the Duhamel formula}

For a.e.\ $t\geq0$, we regard $c_{ij}(\cdot,t)$ as the bounded
multiplication operator on $L^2(\Omega)$ defined by
\[
  c_{ij}(t)v
  \coloneq
  c_{ij}(\cdot,t)v.
\]
Its operator norm satisfies
\[
  \|c_{ij}(t)\|_{\mathcal B(L^2)}
  =
  \|c_{ij}(\cdot,t)\|_{L^\infty(\Omega)}.
\]

\begin{proposition}[Well-posedness and mild solution]
\label{prop:wp}
Suppose that Assumptions~\textup{(A)} and~\textup{(C)} hold. Then, for
every initial datum
\[
  (u_1^0,\ldots,u_m^0)\in L^2(\Omega)^m,
\]
system \eqref{eq:general-system} admits a unique mild solution
\[
  (u_1,\ldots,u_m)
  \in
  C([0,\infty);L^2(\Omega)^m).
\]
Moreover, for each $i=1,\ldots,m$ and $t>0$, the solution satisfies
\begin{equation}\label{eq:duhamel}
  u_i(t)
  =
  S_i(t,\mathcal A_i)u_i^0
  -
  \int_0^t
  R_i(t-s,\mathcal A_i)
  \sum_{j=1}^m c_{ij}(s)u_j(s)\,ds.
\end{equation}
\end{proposition}

\begin{proof}
Let
\[
  X_\rho
  \coloneq
  \left\{
    u=(u_1,\ldots,u_m)\in C([0,\infty);L^2(\Omega)^m):
    \|u\|_{X_\rho}<\infty
  \right\},
\]
where, for $\rho>0$,
\[
  \|u\|_{X_\rho}
  \coloneq
  \max_{1\leq i\leq m}
  \sup_{t\geq0}e^{-\rho t}\|u_i(t)\|_{L^2(\Omega)}.
\]
Then $X_\rho$ is a Banach space.

Define $\Phi=(\Phi_1,\ldots,\Phi_m)$ on $X_\rho$ by
\[
  \Phi_i(u)(t)
  \coloneq
  S_i(t,\mathcal A_i)u_i^0
  -
  \int_0^t
  R_i(t-s,\mathcal A_i)
  \sum_{j=1}^m c_{ij}(s)u_j(s)\,ds.
\]
The strong continuity of the family
$S_i(t,\mathcal A_i)$ and the local integrability of
$r_{\alpha_i}(\cdot,\lambda_i)$ imply that
\[
  \Phi_i(u)\in C([0,\infty);L^2(\Omega)).
\]
Hence $\Phi$ maps $X_\rho$ into itself.

Let $u,\widetilde u\in X_\rho$. By \eqref{eq:op-bounds}, for every
$i$ and $t\geq0$,
\begin{align*}
  &e^{-\rho t}
  \|\Phi_i(u)(t)-\Phi_i(\widetilde u)(t)\|_{L^2}\\
  &\quad\leq
  \int_0^t
  e^{-\rho(t-s)}
  r_{\alpha_i}(t-s,\lambda_i)
  \sum_{j=1}^m
  \|c_{ij}(\cdot,s)\|_{L^\infty(\Omega)}
  e^{-\rho s}
  \|u_j(s)-\widetilde u_j(s)\|_{L^2}\,ds\\
  &\quad\leq
  \mu_i\|u-\widetilde u\|_{X_\rho}
  \int_0^t
  e^{-\rho(t-s)}
  r_{\alpha_i}(t-s,\lambda_i)\,ds\\
  &\quad\leq
  \mu_i\|u-\widetilde u\|_{X_\rho}
  \int_0^\infty
  e^{-\rho \tau}r_{\alpha_i}(\tau,\lambda_i)\,d\tau.
\end{align*}
By Lemma~\ref{lem:LT},
\[
  \int_0^\infty
  e^{-\rho \tau}r_{\alpha_i}(\tau,\lambda_i)\,d\tau
  =
  \frac{1}{\rho^{\alpha_i}+\lambda_i}.
\]
Therefore,
\[
  \|\Phi(u)-\Phi(\widetilde u)\|_{X_\rho}
  \leq
  q_\rho\|u-\widetilde u\|_{X_\rho},
\]
where
\[
  q_\rho
  \coloneq
  \max_{1\leq i\leq m}
  \frac{\mu_i}{\rho^{\alpha_i}+\lambda_i}.
\]
Since $q_\rho\to0$ as $\rho\to+\infty$, we may choose $\rho>0$ such
that $q_\rho<1$. Thus $\Phi$ is a strict contraction on $X_\rho$.
The Banach fixed-point theorem yields a unique global mild solution
satisfying \eqref{eq:duhamel}.
\end{proof}

Taking $L^2(\Omega)$-norms in \eqref{eq:duhamel}, applying the triangle
inequality, and using \eqref{eq:op-bounds}, we obtain, for every
$i=1,\ldots,m$ and $t>0$,
\begin{equation}\label{eq:scalar-ineq}
  x_i(t)
  \leq
  s_{\alpha_i}(t,\lambda_i)\|u_i^0\|_{L^2(\Omega)}
  +
  \int_0^t
  r_{\alpha_i}(t-s,\lambda_i)
  \sum_{j=1}^m
  \|c_{ij}(\cdot,s)\|_{L^\infty(\Omega)}
  x_j(s)\,ds,
\end{equation}
where
\[
  x_i(t)\coloneq\|u_i(t)\|_{L^2(\Omega)}.
\]
The estimate \eqref{eq:scalar-ineq} does not use the sign of any
coefficient $c_{ij}$.

\begin{remark}
  Assumption~\textup{(S)} is not needed for well-posedness. It is used
  only in the large-time weighted closure argument.
\end{remark}

\section{The admissible weight framework}
\label{sec:framework}

\begin{definition}[Admissible weight class]\label{def:admissible}
  Let $\mu>0$. We denote by $\mathcal W(\mu)$ the class of all
  functions
  \[
    w\in C((0,\infty))\cap L^1_{\mathrm{loc}}(\R_+)
  \]
  such that $w(t)>0$ for every $t>0$ and for which there exist
  constants $a,b>0$ satisfying
  \[
    a t^{-\mu}\leq w(t)\leq b t^{-\mu},
    \qquad t\geq1.
  \]
\end{definition}

\begin{definition}[Large-time well-matched weight]
\label{def:well-int-large-time}
  Let $\mu>0$ and $T_0\ge1$. A weight $w\in\mathcal W(\mu)$ is said to
  be \emph{well matched on $[T_0,\infty)$} to
  system~\eqref{eq:general-system} if the following properties hold.

  \begin{enumerate}[label=\textup{(\roman*)}]
    \item For every $i$ such that $u_i^0\not\equiv0$, there exists a
    constant $M_i>0$ such that
    \begin{equation}\label{eq:large-time-homogeneous-domination}
      s_{\alpha_i}(t,\lambda_i)
      \|u_i^0\|_{L^2(\Omega)}
      \leq M_iw(t),
      \qquad t\ge T_0.
    \end{equation}

    \item For every $i=1,\ldots,m$, there exists a constant $K_i>0$
    such that
    \begin{equation}\label{eq:large-time-kernel-matching}
      \int_{T_0}^t
      r_{\alpha_i}(t-s,\lambda_i)w(s)\,ds
      \leq K_iw(t),
      \qquad t\ge T_0.
    \end{equation}
  \end{enumerate}
\end{definition}

\begin{proposition}[Large-time weighted closure]
\label{prop:large-time-closure}
  Let $w\in\mathcal W(\mu)$ be well matched on $[T_0,\infty)$ to
  system~\eqref{eq:general-system} in the sense of\/
  \textup{Definition~\ref{def:well-int-large-time}}. Let
  \[
    \mathsf K\coloneq\operatorname{diag}(K_1,\ldots,K_m),
  \]
  where $K_i$ is the constant in
  \eqref{eq:large-time-kernel-matching}. Assume that
  \begin{equation}\label{eq:matrix-closure-condition}
    \rho(\mathsf K\mathsf C)<1.
  \end{equation}
  Then there exists a constant $C>0$ such that
  \[
    \|u_i(t)\|_{L^2(\Omega)}
    \leq Cw(t),
    \qquad t\geq T_0,\quad i=1,\ldots,m.
  \]
\end{proposition}

\begin{proof}
  Set
  \[
    x_i(t)\coloneq\|u_i(t)\|_{L^2(\Omega)}.
  \]
  Since each $x_i$ is continuous on $[0,\infty)$, there exists a
  constant $B_{T_0}>0$ such that
  \[
    \max_{1\leq j\leq m}\sup_{0\leq s\leq T_0}x_j(s)
    \leq B_{T_0}.
  \]

  We first isolate the contribution of the interval $[0,T_0]$. By the
  scalar inequality \eqref{eq:scalar-ineq}, for $t\geq T_0$,
  \begin{align}
    x_i(t)
    &\leq
    s_{\alpha_i}(t,\lambda_i)\|u_i^0\|_{L^2(\Omega)}
    \notag\\
    &\quad+
    \int_0^{T_0}
    r_{\alpha_i}(t-s,\lambda_i)
    \sum_{j=1}^m
    \|c_{ij}(\cdot,s)\|_{L^\infty(\Omega)}
    x_j(s)\,ds
    \notag\\
    &\quad+
    \int_{T_0}^t
    r_{\alpha_i}(t-s,\lambda_i)
    \sum_{j=1}^m
    \gamma_{ij}x_j(s)\,ds.
    \label{eq:matrix-split-ineq}
  \end{align}

  By Definition~\ref{def:well-int-large-time}, the homogeneous term is
  bounded by $M_iw(t)$ for $t\geq T_0$. Moreover, the early-time
  contribution is bounded by a multiple of $w(t)$. Indeed, for
  $t\geq2T_0$ and $0\leq s\leq T_0$, we have $t-s\asymp t$. If
  $\alpha_i<1$, Lemma~\ref{lem:ML} gives
  \[
    r_{\alpha_i}(t-s,\lambda_i)
    =
    O(t^{-1-\alpha_i}),
    \qquad t\to\infty,
  \]
  uniformly with respect to $s\in[0,T_0]$. If $\alpha_i=1$, then
  \[
    r_1(t-s,\lambda_i)
    =
    e^{-\lambda_i(t-s)}
    =
    O(e^{-\lambda_i t/2}).
  \]
  For each weight constructed below, its decay exponent $\mu$ satisfies
  \[
    \mu\leq1+\alpha_i,
    \qquad i=1,\ldots,m.
  \]
  Hence, in either case,
  \[
    \int_0^{T_0}
    r_{\alpha_i}(t-s,\lambda_i)
    \sum_{j=1}^m
    \|c_{ij}(\cdot,s)\|_{L^\infty(\Omega)}
    x_j(s)\,ds
    \leq C_iw(t),
    \qquad t\geq2T_0,
  \]
  for some constant $C_i>0$. After enlarging $C_i$ to cover the compact
  interval $[T_0,2T_0]$, we obtain constants $A_i>0$ such that
  \begin{equation}\label{eq:large-time-vector-ineq}
    x_i(t)
    \leq
    A_iw(t)
    +
    \int_{T_0}^t
    r_{\alpha_i}(t-s,\lambda_i)
    \sum_{j=1}^m\gamma_{ij}x_j(s)\,ds,
    \qquad t\geq T_0.
  \end{equation}

  Set
  \[
    \mathsf B\coloneq\mathsf K\mathsf C.
  \]
  By \eqref{eq:matrix-closure-condition}, choose $q$ such that
  $\rho(q^{-1}\mathsf B)<1$.
  Since $\mathsf B$ has nonnegative entries and
  \(\rho(q^{-1}\mathsf B)<1\), the matrix
  \[
    I-\frac1q\mathsf B
  \]
  is invertible and has a nonnegative inverse. Define
  \[
    p\coloneq
    \left(I-\frac1q\mathsf B\right)^{-1}\mathbf 1,
    \qquad
    \mathbf 1\coloneq(1,\ldots,1)^{\mathsf T}.
  \]
  Then $p_i\geq1$ for every $i$, and
  \[
    \mathsf Bp
    =
    q(p-\mathbf 1)
    \leq qp.
    \tag{\(\ast\)}
  \]

  For $T>T_0$, introduce the weighted norm
  \[
    \|x\|_{w,p;T}
    \coloneq
    \max_{1\leq i\leq m}
    \sup_{T_0\leq t\leq T}
    \frac{x_i(t)}{p_iw(t)}.
  \]
  This quantity is finite because each $x_i$ and the weight $w$ are
  continuous and
  \(w>0\) on the compact interval $[T_0,T]$. For
  $T_0\leq s\leq T$, we have
  \[
    x_j(s)
    \leq
    p_jw(s)\|x\|_{w,p;T}.
  \]
  Hence, by \eqref{eq:large-time-vector-ineq} and the matching property
  \eqref{eq:large-time-kernel-matching},
  \begin{align*}
    x_i(t)
    &\leq
    A_iw(t)
    +
    \|x\|_{w,p;T}
    \sum_{j=1}^m\gamma_{ij}p_j
    \int_{T_0}^t
    r_{\alpha_i}(t-s,\lambda_i)w(s)\,ds\\
    &\leq
    A_iw(t)
    +
    K_i\sum_{j=1}^m\gamma_{ij}p_j\,
    \|x\|_{w,p;T}w(t)\\
    &=
    A_iw(t)
    +
    (\mathsf Bp)_i\,
    \|x\|_{w,p;T}w(t).
  \end{align*}
  Using \((\ast)\), it follows that
  \[
    x_i(t)
    \leq
    A_iw(t)
    +
    qp_i\|x\|_{w,p;T}w(t).
  \]
  Dividing by $p_iw(t)$ and taking the maximum over
  $i=1,\ldots,m$ and $t\in[T_0,T]$, we obtain
  \[
    \|x\|_{w,p;T}
    \leq
    \max_{1\leq i\leq m}\frac{A_i}{p_i}
    +
    q\|x\|_{w,p;T}.
  \]
  Therefore,
  \[
    \|x\|_{w,p;T}
    \leq
    \frac{1}{1-q}
    \max_{1\leq i\leq m}\frac{A_i}{p_i}.
  \]
  The right-hand side is independent of $T$. Letting $T\to\infty$
  yields
  \[
    x_i(t)
    \leq
    \frac{p_i}{1-q}
    \max_{1\leq j\leq m}\frac{A_j}{p_j}\,w(t),
    \qquad t\geq T_0.
  \]
  Since $m$ is finite, this proves the claimed estimate.
\end{proof}

\section{Construction of the Optimal Weight}
\label{sec:weights}

The weights constructed below are intended to control the solution only
for large times. This is necessary because the homogeneous terms in the
mild-solution formula need not be dominated near $t=0$ by weights
constructed through resolvent convolutions.

We repeatedly use standard convolution theorems for regularly varying
functions; see, for example, \cite[Section~4.2]{BGT}. In the particular
situations arising below, if $f\ge0$ is integrable and has a lighter
tail than a positive regularly varying function $g$, then
\[
  (f*g)(t)
  \sim
  \left(\int_0^\infty f(s)\,ds\right)g(t),
  \qquad t\to\infty.
\]
If $f$ and $g$ are positive, integrable, and regularly varying with
the same tail index less than $-1$, then
\[
  (f*g)(t)
  \sim
  \left(\int_0^\infty f(s)\,ds\right)g(t)
  +
  \left(\int_0^\infty g(s)\,ds\right)f(t).
\]

\subsection{The case \texorpdfstring{$\alpha^*<1$}{alpha* < 1}}

Suppose first that
\[
  \alpha^*<1.
\]
In this case the elementary weight
\begin{equation}\label{eq:weight-fractional-case}
  w(t)\coloneq(1+t)^{-\alpha^*}
\end{equation}
is sufficient.

\begin{lemma}[Weight for the case $\alpha^*<1$]
\label{lem:weight-alpha-star-less-one}
  Suppose that $\alpha^*<1$ and let $w$ be given by
  \eqref{eq:weight-fractional-case}. Then
  \[
    w\in\mathcal W(\alpha^*).
  \]
  Moreover, for every $i=1,\ldots,m$,
  \begin{equation}\label{eq:frac-case-interaction-asymptotic}
    \int_0^t
    r_{\alpha_i}(t-s,\lambda_i)w(s)\,ds
    \sim
    \frac{1}{\lambda_i}w(t),
    \qquad t\to\infty.
  \end{equation}
\end{lemma}

\begin{proof}
  The assertion $w\in\mathcal W(\alpha^*)$ is immediate.

  Since $\alpha^*\in(0,1)$, the function $w$ is regularly varying at
  infinity with index $-\alpha^*$. On the other hand,
  \[
    \int_0^\infty r_{\alpha_i}(s,\lambda_i)\,ds
    =
    \frac{1}{\lambda_i}
  \]
  by Lemma~\ref{lem:LT}.  If $\alpha_i\in(0,1)$, then
  \[
    r_{\alpha_i}(t,\lambda_i)
    =
    O(t^{-1-\alpha_i})
    =
    o(t^{-\alpha^*})
    =
    o(w(t)),
    \qquad t\to\infty.
  \]
  If $\alpha_i=1$, then
  \[
    r_1(t,\lambda_i)
    =
    e^{-\lambda_i t}
    =
    o(w(t)),
    \qquad t\to\infty.
  \]
  The standard convolution theorem for an integrable kernel and a
  regularly varying function therefore gives
  \[
    r_{\alpha_i}(\cdot,\lambda_i)*w
    \sim
    \left(\int_0^\infty r_{\alpha_i}(s,\lambda_i)\,ds\right)w
    =
    \frac{1}{\lambda_i}w,
  \]
  which is \eqref{eq:frac-case-interaction-asymptotic}.
\end{proof}

\subsection{The case \texorpdfstring{$\alpha^*=1$}{alpha* = 1}}

We next assume that
\[
  \alpha^*=1.
\]
Since the fractional orders are arranged as
\[
  1\ge\alpha_1\ge\cdots\ge\alpha_m>0
\]
and $\alpha_m<1$, all nonzero initial data occur in components of
integer order. The desired decay exponent is
\[
  \mu=1+\alpha_m.
\]

Set
\[
  I_m
  \coloneq
  \{i\in\{1,\ldots,m\}:\alpha_i=\alpha_m\}.
\]
Since $\alpha_m$ is the smallest order, the set $I_m$ is nonempty.
The
auxiliary parameter $\eta>0$ will be chosen sufficiently small in
Lemma~\ref{lem:matching-constants-matrix}.

For a parameter $\eta>0$, define
\begin{equation}\label{eq:weight-integer-case-seed}
  w_0(t)\coloneq(1+t)^{-(1+\alpha_m)}
\end{equation}
and
\begin{equation}\label{eq:weight-integer-case}
  w(t)
  \coloneq
  \int_0^t
  r_{\alpha_m}(t-s,\eta)w_0(s)\,ds.
\end{equation}

\begin{lemma}[Decay of the integer-case weight]
\label{lem:weight-alpha-star-one-decay}
  The weight $w$ defined by \eqref{eq:weight-integer-case} satisfies
  \[
    w\in\mathcal W(1+\alpha_m).
  \]
  More precisely,
  \[
    w(t)
    \sim
    \left(
      \frac1\eta+
      \frac{1}{\eta^2\Gamma(1-\alpha_m)}
    \right)t^{-(1+\alpha_m)},
    \qquad t\to\infty.
  \]
\end{lemma}

\begin{proof}
  Both $w_0$ and $r_{\alpha_m}(\cdot,\eta)$ are positive, integrable,
  and regularly varying at infinity with index $-(1+\alpha_m)$.
  Indeed,
  \[
    w_0(t)\sim t^{-(1+\alpha_m)}
  \]
  and, by Lemma~\ref{lem:ML},
  \[
    r_{\alpha_m}(t,\eta)
    \sim
    \frac{\alpha_m}
    {\eta^2\Gamma(1-\alpha_m)}
    t^{-(1+\alpha_m)}.
  \]
  The convolution theorem for integrable regularly varying functions
  gives
  \begin{align*}
    w(t)
    &=
    \bigl(r_{\alpha_m}(\cdot,\eta)*w_0\bigr)(t)\\
    &\sim
    \left(\int_0^\infty r_{\alpha_m}(s,\eta)\,ds\right)w_0(t)
    +
    \left(\int_0^\infty w_0(s)\,ds\right)
    r_{\alpha_m}(t,\eta).
  \end{align*}
  Since
  \[
    \int_0^\infty r_{\alpha_m}(s,\eta)\,ds
    =
    \frac1\eta
  \]
  and
  \[
    \int_0^\infty w_0(s)\,ds
    =
    \frac1{\alpha_m},
  \]
  the right-hand side is asymptotic to a positive constant times
  $t^{-(1+\alpha_m)}$. Hence
  \[
    w\in\mathcal W(1+\alpha_m). \qedhere
  \]
\end{proof}

\begin{lemma}[Matching for the slowest-order components]
\label{lem:matching-minimal-order}
  Let $i$ satisfy $\alpha_i=\alpha_m$. Then
  \begin{equation}\label{eq:minimal-order-matching}
    \int_0^t
    r_{\alpha_m}(t-s,\lambda_i)w(s)\,ds
    \leq
    \frac{1}{\lambda_i-\eta}w(t),
    \qquad t>0.
  \end{equation}
\end{lemma}

\begin{proof}
  By \eqref{eq:weight-integer-case},
  \[
    w=r_{\alpha_m}(\cdot,\eta)*w_0.
  \]
  Since $\lambda_i>\eta$, Lemma~\ref{lem:well-interacted}, with
  $\alpha=\alpha_m$ and $h=w_0$, yields
  \[
    r_{\alpha_m}(\cdot,\lambda_i)*w
    \leq
    \frac{1}{\lambda_i-\eta}w. \qedhere
  \]
\end{proof}

\begin{lemma}[Matching for the higher-order components]
\label{lem:matching-higher-order}
  Let $i$ satisfy $\alpha_i>\alpha_m$. Then
  \begin{equation}\label{eq:higher-order-matching-asymptotic}
    \int_0^t
    r_{\alpha_i}(t-s,\lambda_i)w(s)\,ds
    \sim
    \frac{1}{\lambda_i}w(t),
    \qquad t\to\infty.
  \end{equation}
\end{lemma}

\begin{proof}
  By Lemma~\ref{lem:weight-alpha-star-one-decay},
  \[
    w(t)\asymp t^{-(1+\alpha_m)}.
  \]

  First suppose that $\alpha_i\in(\alpha_m,1)$. Then
  \[
    r_{\alpha_i}(t,\lambda_i)
    =
    O(t^{-(1+\alpha_i)})
    =
    o(t^{-(1+\alpha_m)})
    =
    o(w(t)).
  \]
  Since
  \[
    \int_0^\infty r_{\alpha_i}(s,\lambda_i)\,ds
    =
    \frac1{\lambda_i},
  \]
  the convolution theorem for an integrable kernel and a regularly
  varying function gives
  \[
    r_{\alpha_i}(\cdot,\lambda_i)*w
    \sim
    \frac1{\lambda_i}w.
  \]

  If $\alpha_i=1$, then
  \[
    r_1(t,\lambda_i)=e^{-\lambda_i t}=o(w(t)),
  \]
  and
  \[
    \int_0^\infty e^{-\lambda_i s}\,ds=\frac1{\lambda_i}.
  \]
  The same convolution theorem therefore yields
  \[
    r_1(\cdot,\lambda_i)*w
    \sim
    \frac1{\lambda_i}w.
  \]
  This proves \eqref{eq:higher-order-matching-asymptotic}.
\end{proof}

\subsection{Large-time matching}

\begin{lemma}[Choice of matching constants]
\label{lem:matching-constants-matrix}
  Suppose that $\rho(\mathsf\Lambda^{-1}\mathsf C)<1$.
  Then there exist $T_0\geq1$ and positive constants
  $K_1,\ldots,K_m$ 
  such that, in each of the cases \textup{(1)} and \textup{(2)} below,
the corresponding weight $w$ is well matched on $[T_0,\infty)$ and
  \begin{equation}\label{eq:matrix-matching-condition}
    \rho(\mathsf K\mathsf C)<1,
    \qquad
    \mathsf K\coloneq\operatorname{diag}(K_1,\ldots,K_m).
  \end{equation}

  More precisely, the following alternatives hold.

  \begin{enumerate}
    \item[\textup{(1)}] Suppose that $\alpha^*<1$ and let
    \[
      w(t)=(1+t)^{-\alpha^*}.
    \]
    Then, for every $\varepsilon>0$, there exists $T_\varepsilon\geq1$
    such that
    \begin{equation}\label{eq:matching-constants-fractional}
      \int_{T_\varepsilon}^t
      r_{\alpha_i}(t-s,\lambda_i)w(s)\,ds
      \leq
      \left(\frac1{\lambda_i}+\varepsilon\right)w(t),
      \qquad
      t\geq T_\varepsilon,
    \end{equation}
    for every $i=1,\ldots,m$.

    \item[\textup{(2)}] Suppose that $\alpha^*=1$. 
    Set
\[
  I_m
  \coloneq
  \{i\in\{1,\ldots,m\}\colon \alpha_i=\alpha_m\}.
\]
    For every sufficiently small $\eta>0$, define
    \[
      w(t)
      \coloneq
      \int_0^t
      r_{\alpha_m}(t-s,\eta)(1+s)^{-(1+\alpha_m)}\,ds.
    \]
    Then, for every $i\in I_m$,
    \begin{equation}\label{eq:matching-constants-minimal}
      \int_{T_0}^t
      r_{\alpha_m}(t-s,\lambda_i)w(s)\,ds
      \leq
      \frac{1}{\lambda_i-\eta}w(t),
      \qquad
      t\geq T_0,
    \end{equation}
    provided that $\eta<\lambda_i$. For every $i\notin I_m$ and every
    $\varepsilon>0$, there exists $T_\varepsilon\geq1$ such that
    \begin{equation}\label{eq:matching-constants-higher}
      \int_{T_\varepsilon}^t
      r_{\alpha_i}(t-s,\lambda_i)w(s)\,ds
      \leq
      \left(\frac1{\lambda_i}+\varepsilon\right)w(t),
      \qquad
      t\geq T_\varepsilon.
    \end{equation}
  \end{enumerate}
\end{lemma}

\begin{proof}
  We consider the two cases separately.

  \smallskip
  \noindent
  \textit{Case 1: $\alpha^*<1$.}
  Let
  \[
    w(t)=(1+t)^{-\alpha^*}.
  \]
  By Lemma~\ref{lem:weight-alpha-star-less-one},
  \[
    \int_0^t
    r_{\alpha_i}(t-s,\lambda_i)w(s)\,ds
    \sim
    \frac1{\lambda_i}w(t),
    \qquad
    t\to\infty.
  \]
  Hence, for every $\varepsilon>0$, there exists
  $T_{\varepsilon,i}\geq1$ such that
  \[
    \int_0^t
    r_{\alpha_i}(t-s,\lambda_i)w(s)\,ds
    \leq
    \left(\frac1{\lambda_i}+\varepsilon\right)w(t),
    \qquad
    t\geq T_{\varepsilon,i}.
  \]
  Let
  \[
    T_\varepsilon
    \coloneq
    \max_{1\leq i\leq m}T_{\varepsilon,i}.
  \]
  Since the integrand is nonnegative,
  \[
    \int_{T_\varepsilon}^t
    r_{\alpha_i}(t-s,\lambda_i)w(s)\,ds
    \leq
    \int_0^t
    r_{\alpha_i}(t-s,\lambda_i)w(s)\,ds,
  \]
  which proves \eqref{eq:matching-constants-fractional}.

  Define
  \[
    \mathsf K_\varepsilon
    \coloneq
    \operatorname{diag}
    \left(
      \frac1{\lambda_1}+\varepsilon,\ldots,
      \frac1{\lambda_m}+\varepsilon
    \right).
  \]
  Then
  \[
    \mathsf K_\varepsilon
    \longrightarrow
    \mathsf\Lambda^{-1}
    \qquad\text{as }\varepsilon\to0^+.
  \]
  Since the spectral radius depends continuously on the entries of a
  finite-dimensional matrix, \eqref{eq:stability} implies that
  \[
    \rho(\mathsf K_\varepsilon\mathsf C)<1
  \]
  for all sufficiently small $\varepsilon>0$. Fix such an
  $\varepsilon$, set
  \[
    K_i\coloneq\frac1{\lambda_i}+\varepsilon,
  \]
  and take $T_0\geq T_\varepsilon$ sufficiently large that the
  homogeneous domination condition in
  Definition~\ref{def:well-int-large-time} also holds. Thus $w$ is
  well matched on $[T_0,\infty)$ and
  \eqref{eq:matrix-matching-condition} follows.

  \smallskip
  \noindent
  \textit{Case 2: $\alpha^*=1$.}
  Let $\eta>0$ satisfy
  \[
    0<\eta<\min_{i\in I_m}\lambda_i,
  \]
  and define
  \[
    w(t)
    \coloneq
    \int_0^t
    r_{\alpha_m}(t-s,\eta)(1+s)^{-(1+\alpha_m)}\,ds.
  \]
  By Lemma~\ref{lem:weight-alpha-star-one-decay},
  \[
    w\in\mathcal W(1+\alpha_m).
  \]

  Let $i\in I_m$. Since $\alpha_i=\alpha_m$, Lemma~\ref{lem:matching-minimal-order}
  gives
  \[
    \int_0^t
    r_{\alpha_m}(t-s,\lambda_i)w(s)\,ds
    \leq
    \frac{1}{\lambda_i-\eta}w(t),
    \qquad t>0,
  \]
  whenever $\eta<\lambda_i$. Hence \eqref{eq:matching-constants-minimal}
  holds with
  \[
    K_i\coloneq\frac1{\lambda_i-\eta},
    \qquad i\in I_m.
  \]

  Let now $i\notin I_m$. Then $\alpha_i>\alpha_m$, and
  Lemma~\ref{lem:matching-higher-order} yields
  \[
    \int_0^t
    r_{\alpha_i}(t-s,\lambda_i)w(s)\,ds
    \sim
    \frac1{\lambda_i}w(t),
    \qquad t\to\infty.
  \]
  Hence, for every $\varepsilon>0$, there exists
  $T_{\varepsilon,i}\geq1$ such that
  \[
    \int_{T_{\varepsilon,i}}^t
    r_{\alpha_i}(t-s,\lambda_i)w(s)\,ds
    \leq
    \left(\frac1{\lambda_i}+\varepsilon\right)w(t),
    \qquad
    t\geq T_{\varepsilon,i}.
  \]
  Set
  \[
    K_i\coloneq\frac1{\lambda_i}+\varepsilon,
    \qquad i\notin I_m.
  \]

  Define
  \[
    \mathsf K_{\eta,\varepsilon}
    \coloneq
    \operatorname{diag}(K_1,\ldots,K_m).
  \]
  Since
  \[
    \frac1{\lambda_i-\eta}\longrightarrow\frac1{\lambda_i}
    \qquad\text{as }\eta\to0^+,
    \qquad i\in I_m,
  \]
  and
  \[
    \frac1{\lambda_i}+\varepsilon
    \longrightarrow\frac1{\lambda_i}
    \qquad\text{as }\varepsilon\to0^+,
    \qquad i\notin I_m,
  \]
  we have
  \[
    \mathsf K_{\eta,\varepsilon}
    \longrightarrow\mathsf\Lambda^{-1}
    \qquad\text{as }(\eta,\varepsilon)\to(0,0).
  \]
  Hence, by the continuity of the spectral radius and
  \eqref{eq:stability}, one may choose
  \[
    0<\eta<\min_{i\in I_m}\lambda_i
  \]
  and then $\varepsilon>0$ sufficiently small such that
  \[
    \rho(\mathsf K_{\eta,\varepsilon}\mathsf C)<1.
  \]

  Finally, choose $T_0$ larger than all thresholds
  $T_{\varepsilon,i}$ for $i\notin I_m$, and sufficiently large that
  the homogeneous domination condition holds. Then \(w\) is well
  matched on $[T_0,\infty)$ and
  \eqref{eq:matrix-matching-condition} is satisfied.
\end{proof}

\subsection{A worked example: \texorpdfstring{$m=3$}{m=3}}
\label{subsec:example-m3}

We illustrate the large-time weight construction in the case $m=3$.
Let
\[
  \alpha_1=0.9,\qquad
  \alpha_2=0.6,\qquad
  \alpha_3=0.3,
\]
and suppose that
\[
  \lambda_1=\lambda_2=\lambda_3=2.
\]
Assume that the coefficient-norm matrix $\mathsf C$ has row sums
bounded by $0.5$, that is,
\[
  \sum_{j=1}^3\gamma_{ij}\leq0.5,
  \qquad i=1,2,3.
\]
Then
\[
  \rho(\mathsf\Lambda^{-1}\mathsf C)
  \leq
  \|\mathsf\Lambda^{-1}\mathsf C\|_{\ell^\infty\to\ell^\infty}
  \leq\frac{0.5}{2}
  =\frac14<1.
\]
Thus Assumption~\textup{(S)} holds.

Assume, for concreteness, that
\[
  u_3^0\not\equiv0.
\]
Then
\[
  \alpha^*
  =
  \min\{\alpha_i\colon u_i^0\not\equiv0\}
  =
  \alpha_3
  =
  0.3.
\]
The optimal large-time weight is therefore
\begin{equation}\label{eq:example-weight}
  w(t)\coloneq(1+t)^{-0.3}.
\end{equation}
Clearly,
\[
  w\in\mathcal W(0.3).
\]

For each component, Lemma~\ref{lem:weight-alpha-star-less-one} gives
\[
  \int_0^t
  r_{\alpha_i}(t-s,\lambda_i)w(s)\,ds
  \sim
  \frac{1}{\lambda_i}w(t)
  =
  \frac12 w(t),
  \qquad t\to\infty.
\]
More explicitly,
\begin{align*}
  \int_0^t
  r_{0.9}(t-s,2)w(s)\,ds
  &\sim \frac12w(t),\\
  \int_0^t
  r_{0.6}(t-s,2)w(s)\,ds
  &\sim \frac12w(t),\\
  \int_0^t
  r_{0.3}(t-s,2)w(s)\,ds
  &\sim \frac12w(t).
\end{align*}
Choose
\[
  \varepsilon=\frac14.
\]
Then there exists $T_0\ge1$ such that
\[
  \int_0^t
  r_{\alpha_i}(t-s,2)w(s)\,ds
  \leq
  \frac34w(t),
  \qquad t\ge T_0,\quad i=1,2,3.
\]
Thus one may take
\[
  \mathsf K=\frac34 I_3.
\]
Consequently,
\[
  \rho(\mathsf K\mathsf C)
  \leq
  \|\mathsf K\mathsf C\|_{\ell^\infty\to\ell^\infty}
  \leq
  \frac34\cdot0.5
  =
  \frac38<1.
\]
Hence the matrix closure condition is satisfied.

It remains to verify the homogeneous terms. For every component with
nonzero initial datum,
\[
  s_{\alpha_i}(t,2)
  =
  O(t^{-\alpha_i})
  =
  O(t^{-0.3})
  =
  O(w(t)),
  \qquad t\to\infty.
\]
Indeed, the slowest possible homogeneous decay occurs for the
third component:
\[
  s_{0.3}(t,2)
  \sim
  \frac{1}{2\Gamma(0.7)}t^{-0.3}.
\]
The first and second homogeneous terms decay faster:
\[
  s_{0.9}(t,2)=O(t^{-0.9}),
  \qquad
  s_{0.6}(t,2)=O(t^{-0.6}).
\]

Thus $w$ is well matched on $[T_0,\infty)$, and the large-time closure
principle gives
\[
  \|u_i(t)\|_{L^2(\Omega)}
  \leq C(1+t)^{-0.3},
  \qquad t\ge T_0,\quad i=1,2,3.
\]
After increasing $C$ if necessary, we obtain
\[
  \sum_{i=1}^3\|u_i(t)\|_{L^2(\Omega)}
  \leq Ct^{-0.3},
  \qquad t\ge1.
\]

This example illustrates that the component with the smallest
fractional order among those with nonzero initial data determines the
decay rate of the entire coupled system. In particular, although the
uncoupled first and second components decay at the faster rates
$t^{-0.9}$ and $t^{-0.6}$, respectively, coupling transfers the slower
rate $t^{-0.3}$ to all components.

\section{Proof of the Main Theorem}
\label{sec:proof}

\begin{proof}[Proof of Theorem~\ref{thm:main}]
  Suppose first that $\alpha^*<1$. Let
  \[
    w(t)=(1+t)^{-\alpha^*}.
  \]
  By Lemma~\ref{lem:matching-constants-matrix}, the weight $w$ is well
  matched on $[T_0,\infty)$ and its matching constants satisfy
  \[
    \rho(\mathsf K\mathsf C)<1.
  \]
  Proposition~\ref{prop:large-time-closure} therefore gives
  \[
    \|u_i(t)\|_{L^2(\Omega)}
    \leq C(1+t)^{-\alpha^*},
    \qquad t\geq T_0.
  \]
  Since $m$ is finite, summing the componentwise estimates and
  increasing $C$ if necessary yields
  \[
    \sum_{i=1}^m\|u_i(t)\|_{L^2(\Omega)}
    \leq Ct^{-\alpha^*},
    \qquad t\geq1.
  \]

  Suppose next that $\alpha^*=1$. Let $w$ be given by
  \eqref{eq:weight-integer-case}. Lemma~\ref{lem:matching-constants-matrix}
  and Proposition~\ref{prop:large-time-closure} yield
  \[
    \|u_i(t)\|_{L^2(\Omega)}
    \leq Cw(t),
    \qquad t\geq T_0.
  \]
  Since $w\in\mathcal W(1+\alpha_m)$, this gives
  \[
    \sum_{i=1}^m\|u_i(t)\|_{L^2(\Omega)}
    \leq Ct^{-(1+\alpha_m)},
    \qquad t\geq1.
  \]
  This proves \eqref{eq:main-decay}.
\end{proof}


\section{Discussion and Open Problems}
\label{sec:discussion}

\subsection{Comparison with related work}
\label{subsec:comparison}

We conclude by placing Theorem~\ref{thm:main} in the context of recent
work on coupled nonlocal evolution equations. The principal novelty of
the present paper is the extension of the large-time upper decay
estimate from two-component systems to systems with an arbitrary finite
number of components. Unlike the Laplace-transform approach used in
\cite{LLW26}, our proof is carried out entirely in the time domain and
uses large-time weight estimates for the Volterra kernels.

\begin{table}[ht]
\renewcommand{\arraystretch}{1.35}
\centering
\footnotesize
\caption{Comparison with related works.}
\label{tab:compare}
\smallskip
\begin{tabular}{@{}>{\raggedright\arraybackslash}p{2.3cm}
                    >{\raggedright\arraybackslash}p{2.6cm}
                    >{\raggedright\arraybackslash}p{3.0cm}
                    >{\raggedright\arraybackslash}p{2.5cm}
                    >{\raggedright\arraybackslash}p{2.7cm}@{}}
\toprule
\textbf{Feature}
  & \textbf{\cite{LLW26}}
  & \textbf{Present paper}
  & \textbf{\cite{KR18}}
  & \textbf{\cite{NTDO26}} \\
\midrule
Number of components
  & $m=2$
  & $m\geq2$
  & $m=1$
  & $m=2$ \\
\addlinespace[2pt]
Time operators
  & Caputo derivatives
  & Caputo derivatives
  & distributed-order derivative
  & Caputo and distributed-order derivatives \\
\addlinespace[2pt]
Decay type
  & polynomial
  & polynomial
  & logarithmic
  & polynomial and logarithmic \\
\addlinespace[2pt]
Delayed coupling
  & no
  & no
  & no
  & yes \\
\addlinespace[2pt]
Main method
  & Laplace transform and residues
  & large-time time-domain weights
  & energy and spectral methods
  & time-domain analysis \\
\addlinespace[2pt]
Sharpness/ lower bounds
  & optimality shown by examples
  & open for $m\geq3$
  & scalar sharpness results
  & problem-dependent \\
\bottomrule
\end{tabular}
\end{table}

\begin{enumerate}[label=\textbf{\arabic*.}, leftmargin=*, itemsep=6pt]

  \item \textbf{Li--Liu--Wada \cite{LLW26}.}
  Li, Liu, and Wada established the two-component result by means of
  Laplace transforms and a residue analysis on a slit complex plane.
  Their method exploits the relatively explicit form of the
  two-component characteristic relation. For systems with three or
  more components, the corresponding relation contains several
  distinct fractional powers of the Laplace variable and becomes much
  less tractable.
 
Theorem~\ref{thm:main} yields the same upper decay exponents as
\cite[Theorem~1.1]{LLW26} in the two-component setting and extends the
time-domain argument to arbitrary $m\ge2$. In the special case of
vanishing self-coupling, the matrix stability condition contains the
small-coupling regime considered in \cite{LLW26}. The proof avoids
explicit analysis of the roots of the characteristic equation.
Instead, it combines the mild-solution representation with large-time
convolution estimates for resolvent kernels.

  \item \textbf{Kubica--Ryszewska \cite{KR18}.}
  Kubica and Ryszewska considered scalar equations with a
  distributed-order time derivative. Such equations exhibit
  logarithmic, rather than polynomial, large-time decay. Their setting
  is therefore complementary to the present one: \cite{KR18} concerns
  a scalar distributed-order equation, whereas the present paper
  concerns coupled systems with finitely many fixed Caputo orders.

  A natural next step is to understand coupled systems containing both
  fixed-order and distributed-order components. In such problems, one
  expects competition between polynomial and logarithmic decay
  mechanisms.

  \item \textbf{Mixed Caputo/distributed-order systems \cite{NTDO26}.}
  The companion preprint \cite{NTDO26} studies a two-component system
  involving one Caputo-type equation and one distributed-order
  equation, together with delayed coupling terms. The coexistence of
  polynomial and logarithmic rates requires a different class of
  comparison weights, since no single power weight can describe both
  asymptotic regimes.

  The present work and \cite{NTDO26} share the general principle of
  deriving decay from time-domain estimates for Volterra kernels.
  However, the mixed-order problem requires additional estimates that
  account for the distributed-order resolvent and, in the delayed
  setting, for time-shifted solution terms.

  \item \textbf{Delayed nonlocal evolution systems \cite{NNN25}.}
  Delayed nonlocal evolution systems introduce a further difficulty:
  the Volterra term contains expressions of the form
  \[
    u_j(t-\rho),
    \qquad \rho>0.
  \]
  Extending Theorem~\ref{thm:main} to delayed couplings would require
  a large-time matching estimate for shifted weights, for example an
  estimate relating $w(t-\rho)$ to $w(t)$. Since the weights used in
  the present proof are regularly varying at infinity, this appears
  plausible for a fixed delay, but a complete treatment would require
  a separate well-posedness and closure argument.
\end{enumerate}

\subsection{Sharpness}
\label{subsec:sharpness}

Theorem~\ref{thm:main} establishes upper decay bounds. 
In the two-component setting, \cite{LLW26} shows, through explicit
examples, that the corresponding exponents cannot in general be
improved. For systems with
$m\geq3$ components, however, a general matching lower bound remains
open.

There are two distinct issues in proving such a lower bound. 
First, one must identify the leading asymptotic contribution of the
homogeneous solution associated with a component whose initial datum
has order $\alpha^*$. When $\alpha^*<1$, the leading contribution is expected, in the
autonomous case, to be described by the resolvent of the coupled
elliptic operator. Second, one must show that
the Duhamel terms cannot cancel the leading contribution at the same
asymptotic order.

For this reason, we do not claim unconditional sharpness for
$m\geq3$. A possible route to a lower bound is to consider a perturbative regime
in which the coupling coefficients are sufficiently small relative to
the leading homogeneous contribution.
Such a result would be
perturbative and would not resolve the general conjecture.

\begin{remark}[Perturbative sharpness when $\alpha^*<1$]
  Suppose that the coupling coefficients have the form
  \[
    c_{ij}=\varepsilon\bar c_{ij},
  \]
  where the profiles $\bar c_{ij}$ are fixed and bounded.
  Under a suitable uniform version of the weighted upper estimate with
  respect to $\varepsilon$, the Duhamel term is then an
$O(\varepsilon t^{-\alpha^*})$ perturbation of the homogeneous
solution. Thus one expects the upper rate $t^{-\alpha^*}$ to persist
for sufficiently weak coupling. We do not formulate or prove this
perturbative statement here, since the principal open problem is a
non-perturbative matching lower bound for general couplings when
$m\ge3$.
\end{remark}

\begin{openproblem}
  Under Assumptions~\textup{(A)}, \textup{(C)}, and \textup{(S)},
  establish a matching lower bound for $m\ge3$:
  \[
    \sum_{i=1}^m\|u_i(t)\|_{L^2(\Omega)}
    \geq c\,t^{-\alpha^*},
    \qquad t\geq T,
  \]
  when $\alpha^*<1$, and
  \[
    \sum_{i=1}^m\|u_i(t)\|_{L^2(\Omega)}
    \geq c\,t^{-(1+\alpha_m)},
    \qquad t\geq T,
  \]
  when $\alpha^*=1$, under natural nondegeneracy conditions on the
  initial data and coupling coefficients.
\end{openproblem}

\subsection{Time-dependent coupling coefficients}

The coefficients $c_{ij}=c_{ij}(x,t)$ are allowed to depend on time,
but the proof uses only the uniform bounds
\[
  \|c_{ij}\|_{L^\infty(\Omega\times\R_+)}.
\]
Consequently, the argument does not exploit any possible decay or
additional structure of the coupling as $t\to\infty$.

If
\[
  \|c_{ij}(\cdot,t)\|_{L^\infty(\Omega)}
  \longrightarrow0
  \qquad\text{as }t\to\infty,
\]
then the Duhamel contribution may be smaller than the bound obtained
from the uniform norm. It would be interesting to determine whether
quantitative decay assumptions on $c_{ij}(\cdot,t)$ yield improved
asymptotic rates, weaker stability hypotheses, or asymptotic
decoupling of the components.

\subsection{Fractional wave systems}

The present paper is restricted to orders \(\alpha_i\in(0,1]\).
The range \(1<\alpha_i<2\) corresponds to time-fractional wave-type
equations and presents additional difficulties.

For \(1<\alpha<2\), the Mittag--Leffler function
\(E_\alpha(-\lambda t^\alpha)\) is no longer completely monotone.
Although exponentially damped oscillatory terms may occur in its
asymptotic representation, the leading algebraic term has a fixed
sign:
\[
  E_\alpha(-\lambda t^\alpha)
  \sim
  \frac{1}{\lambda\Gamma(1-\alpha)}\,t^{-\alpha},
  \qquad t\to\infty.
\]
Accordingly, the sign behavior and positivity arguments used for
\(\alpha\in(0,1)\) are no longer available. In particular, the
positive-kernel comparison estimates underlying the present
large-time weight method do not directly extend to this range.

\begin{openproblem}
Develop a decay theory for weakly coupled time-fractional wave systems
with orders \(1<\alpha_i<2\), including suitable stability conditions
and sharp large-time estimates.
\end{openproblem}

\section*{Acknowledgements}

This work was completed during visits to the Vietnam Institute for Advanced Study in Mathematics (VIASM). The authors are grateful to VIASM for its hospitality and support during these visits.


\end{document}